\documentclass[12pt,reqno]{article}
\usepackage{amssymb,amscd,amsmath,amsthm,color}
\newtheorem{thm}{Theorem}[section]
\newtheorem{prop}[thm]{Proposition}
\newtheorem{lemma}[thm]{Lemma}
\newtheorem{cor}[thm]{Corollary}

\newtheorem{remark}[thm]{Remark}
\newtheorem{example}[thm]{Example}
\newtheorem{problem}[thm]{Problem}
\numberwithin{equation}{section}

\def\eps{\varepsilon}
\def\bC{\mathbb{C}}
\def\bR{\mathbb{R}}
\def\Tr{\mathrm{Tr}\,}
\def\bM{\mathbb{M}}
\def\bN{\mathbb{N}}
\def\id{\mathrm{id}}
\def\OM{\mathrm{OM}}
\def\ffi{\varphi}

\def\bx{\mathbf{x}}
\def\by{\mathbf{y}}

\begin{document}
\baselineskip=16pt
\allowdisplaybreaks

\centerline{\LARGE A generalization of Araki's log-majorization}
\bigskip
\bigskip
\centerline{\Large
Fumio Hiai\footnote{{\it E-mail address:} hiai.fumio@gmail.com}}

\medskip
\begin{center}
$^1$\,Tohoku University (Emeritus), \\
Hakusan 3-8-16-303, Abiko 270-1154, Japan
\end{center}

\medskip
\begin{abstract}
We generalize Araki's log-majorization to the log-convexity theorem for the eigenvalues of
$\Phi(A^p)^{1/2}\Psi(B^p)\Phi(A^p)^{1/2}$ as a function of $p\ge0$, where $A,B$ are positive
semidefinite matrices and $\Phi,\Psi$ are positive linear maps between matrix algebras.
A similar generalization of the log-majorization of Ando-Hiai type is given as well.


\end{abstract}

\section{Introduction}

The Lieb-Thirring inequality \cite{LT} and its extension by Araki \cite{Ar} are regarded as a
strengthening of the celebrated Golden-Thompson trace inequality, which can be written, as
explicitly stated in \cite{AH}, in terms of log-majorization
\begin{equation}\label{F-1.1}
(A^{1/2}BA^{1/2})^r\prec_{(\log)}A^{r/2}B^rA^{r/2},\qquad r\ge1,
\end{equation}
for matrices $A,B\ge0$. Here, for $n\times n$ matrices $X,Y\ge0$, the log-majorization
$X\prec_{(\log)}Y$ means that
$$
\prod_{i=1}^k\lambda_i(X)\le\prod_{i=1}^k\lambda_i(Y),\qquad k=1,\dots,n
$$
with equality for $k=n$, where $\lambda_1(X)\ge\dots\ge\lambda_n(X)$ are the eigenvalues of $X$
arranged in decreasing order and counting multiplicities. The weak log-majorization
$X\prec_{w(\log)}Y$ is referred to when the last equality is not imposed. A concise survey of
majorization for matrices is found in, e.g., \cite{An2} (also \cite{Hi1,Hi2}).

In the present paper we generalize the log-majorization in \eqref{F-1.1} to the log-convexity
of the function
$$
p\in[0,\infty)\longmapsto\lambda\bigl(\Phi(A^p)^{1/2}\Psi(B^p)\Phi(A^p)^{1/p}\bigr)
$$
in the sense of weak log-majorization order, involving positive linear maps $\Phi,\Psi$ between
matrix algebras. More precisely, in Theorem \ref{T-3.1} of Section 3, we prove the weak
log-majorization
\begin{align}
&\lambda\bigl(\Phi(A^{p_\alpha})^{1/2}\Psi(B^{p_\alpha})\Phi(A^{p_\alpha})^{1/2}\bigr)
\nonumber\\
&\quad\prec_{w(\log)}
\lambda^{1-\alpha}\bigl(\Phi(A^{p_0})^{1/2}\Psi(B^{p_0})\Phi(A^{p_0})^{1/2}\bigr)
\lambda^\alpha\bigl(\Phi(A^{p_1})^{1/2}\Psi(B^{p_1})\Phi(A^{p_1})^{1/2}\bigr), \label{F-1.2}
\end{align}
where $p_\alpha:=(1-\alpha)p_0+\alpha p_1$ for $0\le\alpha\le1$. In particular, when
$\Phi=\Psi=\id$ and $(p_0,p_1)=(0,1)$, \eqref{F-1.2} reduces to
\begin{equation}\label{F-1.3}
\lambda(A^{\alpha/2}B^\alpha A^{\alpha/2})\prec_{w(\log)}
\lambda^\alpha(A^{1/2}BA^{1/2}),\qquad0\le\alpha\le1,
\end{equation}
which is equivalent to \eqref{F-1.1} by letting $\alpha=1/r$ and replacing $A,B$ with
$A^r,B^r$. In Section 2 we show an operator norm inequality in a more general setting
by a method using operator means. In Section 3 we extend this inequality to the weak
log-majorization \eqref{F-1.2} by applying the well-known antisymmetric tensor power technique.

The recent paper of Bourin and Lee \cite{BL} contains, as a consequence of their joint
log-convexity theorem for a two-variable norm function, the weak log-majorization 
$$
(A^{1/2}Z^*BZA^{1/2})^r\prec_{w(\log)}A^{r/2}Z^*B^rZA^{r/2},\qquad r\ge1,
$$
which is closely related to ours, as explicitly mentioned in Remark \ref{R-3.6} of Section 3.

The complementary Golden-Thompson inequality was first shown in \cite{HP} and then it was
extended in \cite{AH} to the log-majorization
$$
A^r\,\#_\alpha\,B^r\prec_{(\log)}(A\,\#_\alpha\,B)^r,\qquad r\ge1,
$$
where $\#_\alpha$ is the weighted geometric mean for $0\le\alpha\le1$. In a more recent paper
\cite{Wa} the class of operator means $\sigma$ for which
$\lambda_1(A^r\,\sigma\,B^r)\le\lambda_1^r(A\,\sigma\,B)$ holds for all $r\ge1$ was
characterized in terms of operator monotone functions representing $\sigma$. In Section 4 of
the paper, we show some generalizations of these results in \cite{AH,Wa} in a somewhat similar
way to that of Araki's log-majorization in Sections 2 and 3.

\section{Operator norm inequalities}

For $n\in\bN$ we write $\bM_n$ for the $n\times n$ complex matrix algebra and $\bM_n^+$ for the
$n\times n$ positive semidefinite matrices. For $A\in\bM_n$ we write $A\ge0$ if $A\in\bM_n^+$,
and $A>0$ if $A$ is positive definite, i.e., $A\ge0$ and $A$ is invertible. The operator norm
and the usual trace of $A\in\bM_n$ is denoted by $\|A\|_\infty$ and $\Tr A$, respectively.

We denote by $\OM_{+,1}$ the set of non-negative operator monotone functions $f$ on $[0,\infty)$
such that $f(1)=1$. In theory of operator means due to Kubo and Ando \cite{KA}, a main result
says that each operator mean $\sigma$ is associated with an $f\in\OM_{+,1}$ in such a way that
$$
A\,\sigma\,B:=A^{1/2}f(A^{-1/2}BA^{-1/2})A^{1/2}
$$
for $A,B\in\bM_n^+$ with $A>0$, which is further extended to general $A,B\in\bM_n^+$ as
$$
A\,\sigma\,B:=\lim_{\eps\searrow0}(A+\eps I_n)\,\sigma\,(B+\eps I_n).
$$
We write $\sigma_f$ for the operator mean associated with $f\in\OM_{+,1}$. For $0\le\alpha\le1$,
the operator mean corresponding to the function $x^\alpha$ in $\OM_{+,1}$ is called the
{\it weighted geometric mean} denoted by $\#_\alpha$; more explicitly,
$$
A\,\#_\alpha\,B=A^{1/2}(A^{-1/2}BA^{-1/2})^\alpha A^{1/2}
$$
for $A,B\in\bM_n^+$ with $A>0$. The case $\alpha=1/2$ is the {\it geometric mean} $\#$, first
introduced by Pusz and Woronowicz \cite{PW}. Let $\sigma_f^*$ be the adjoint of $\sigma_f$,
i.e., the operator mean corresponding to $f^*\in\OM_{+,1}$ defined as $f^*(x):=f(x^{-1})^{-1}$,
$x>0$.

A linear map $\Phi:\bM_n\to\bM_l$ is said to be positive if $\Phi(A)\in\bM_l^+$ for all
$A\in\bM_n^+$, which is furthermore said to be strictly positive if $\Phi(I_n)>0$, that is,
$\Phi(A)>0$ for all $A\in\bM_n$ with $A>0$. In the rest of the paper, we throughout assume
that $\Phi:\bM_n\to\bM_l$ and $\Psi:\bM_m\to\bM_l$ are positive linear maps. Recall the
well-known fact, essentially due to Ando \cite{An1}, that
$$
\Phi(A\,\sigma\,B)\le\Phi(A)\,\sigma\,\Phi(B)
$$
for all $A,B\in\bM_n^+$ and for any operator mean $\sigma$. This will be repeatedly used
without reference in the sequel.

For non-negative functions $\ffi_0$ and $\ffi_1$ on $[0,\infty)$ a new non-negative function
$\ffi:=\ffi_0\,\sigma_f\,\ffi_1$ on $[0,\infty)$ is defined as
$$
\ffi(x)=\ffi_0(x)\,\sigma_f\,\ffi_1(x)
=\lim_{\eps\searrow0}(\ffi_0(x)+\eps)f\biggl({\ffi_1(x)+\eps\over\ffi_0(x)+\eps}\biggr),
\qquad x\in[0,\infty).
$$

\begin{prop}\label{P-2.1}
Let $f\in\OM_{+,1}$. Let $\ffi_0$ and $\ffi_1$ be arbitrary non-negative functions on
$[0,\infty)$ and define the functions $\ffi:=\ffi_0\,\sigma_f\,\ffi_1$ and
$\widetilde\ffi:=\ffi_0\,\sigma_f^*\,\ffi_1$ on $[0,\infty)$ as above. Then for every
$A\in\bM_n^+$ and $B\in\bM_m^+$,
\begin{align*}
&\big\|\Phi(\widetilde\ffi(A))^{1/2}\Psi(\ffi(B))
\Phi(\widetilde\ffi(A))^{1/2}\big\|_\infty \\
&\quad\le\max\bigl\{\big\|\Phi(\ffi_0(A))^{1/2}\Psi(\ffi_0(B))
\Phi(\ffi_0(A))^{1/2}\big\|_\infty,
\big\|\Phi(\ffi_1(A))^{1/2}\Psi(\ffi_1(B))
\Phi(\ffi_1(A))^{1/2}\big\|_\infty\bigr\}.
\end{align*}
\end{prop}

\begin{proof}
Letting
$$
\gamma_k:=\big\|\Phi(\ffi_k(A))^{1/2}\Psi(\ffi_k(B))
\Phi(\ffi_k(A))^{1/2}\big\|_\infty,\qquad k=0,1,
$$
we may prove that
\begin{equation}\label{F-2.1}
\Phi(\widetilde\ffi(A))^{1/2}\Psi(\ffi(B))\Phi(\widetilde\ffi(A))^{1/2}\le
\max\{\gamma_0,\gamma_1\}I_l.
\end{equation}
First, assume that $\Phi$ and $\Psi$ are strictly positive and $\ffi_0(x),\ffi_1(x)>0$ for any
$x\ge0$. Then $\gamma_0,\gamma_1>0$, and we have
$$
\Psi(\ffi_k(B))\le\gamma_k\Phi(\ffi_k(A))^{-1},\qquad k=0,1.
$$
Since $\ffi(B)=\ffi_0(B)\,\sigma_f\,\ffi_1(B)$ and
$\widetilde\ffi(A)=\ffi_0(A)\,\sigma_f^*\,\ffi_1(A)$, by the joint monotonicity of $\sigma_f$
we have
\begin{align}
\Psi(\ffi(B))&\le\Psi(\ffi_0(B))\,\sigma_f\,\Psi(\ffi_1(B)) \nonumber\\
&\le\bigl(\gamma_0\Phi(\ffi_0(A))^{-1}\bigr)\,\sigma_f\,
\bigl(\gamma_1\Phi(\ffi_1(A))^{-1}\bigr) \nonumber\\
&\le\max\{\gamma_0,\gamma_1\}\bigl\{\Phi(\ffi_0(A))\,\sigma_f^*\,\Phi(\ffi_1(A))\bigr\}^{-1}
\nonumber\\
&\le\max\{\gamma_0,\gamma_1\}\Phi\bigl(\ffi_0(A)\,\sigma_f^*\,\ffi_1(A)\bigr)^{-1}
\nonumber\\
&=\max\{\gamma_0,\gamma_1\}\Phi(\widetilde\ffi(A))^{-1}, \label{F-2.2}
\end{align}
which implies \eqref{F-2.1} under the assumptions given above.

For the general case, for every $\eps>0$ we define a strictly positive $\Phi_\eps:\bM_n\to\bM_l$
by
$$
\Phi_\eps(X):=\Phi(X)+\eps\Tr(X)I_l.
$$
and similarly $\Psi_\eps:\bM_m\to\bM_l$. Moreover let $\ffi_{k,\eps}(x):=\ffi_k(x)+\eps$,
$k=0,1$, for $x\ge0$, and $\ffi_\eps:=\ffi_{0,\eps}\,\sigma_f\,\ffi_{1,\eps}$,
$\widetilde\ffi_\eps:=\ffi_{0,\eps}\,\sigma_f^*\,\ffi_{1,\eps}$. By the above case we then
have
\begin{equation}\label{F-2.3}
\Phi_\eps(\widetilde\ffi_\eps(A))^{1/2}\Psi_\eps(\ffi_\eps(B))
\Phi_\eps(\widetilde\ffi_\eps(A))^{1/2}
\le\max\{\gamma_{0,\eps},\gamma_{1,\eps}\}I_l,
\end{equation}
where
$$
\gamma_{k,\eps}:=\big\|\Phi_\eps(\ffi_{k,\eps}(A))^{1/2}\Psi_\eps(\ffi_{k,\eps}(B))
\Phi_\eps(\ffi_{k,\eps}(A))^{1/2}\big\|_\infty,\qquad k=0,1.
$$
Since $\widetilde\ffi_\eps(A)\to\widetilde\ffi(A)$, $\ffi_\eps(B)\to\ffi(B)$ and
$\gamma_{k,\eps}\to\gamma_k$, $k=0,1$,  as $\eps\searrow0$, we have \eqref{F-2.1} in the
general case by taking the limit of \eqref{F-2.3}.
\end{proof}

For non-negative functions $\ffi_0,\ffi_1$ the function $\ffi_0^{1-\alpha}\ffi_1^\alpha$ with
$0\le\alpha\le1$ is often called the {\it geometric bridge} of $\ffi_0,\ffi_1$, for which we
have

\begin{prop}\label{P-2.2}
Let $\ffi_0,\ffi_1$ be arbitrary non-negative functions on $[0,\infty)$ and $0\le\alpha\le1$.
Define $\ffi_\alpha(x):=\ffi_0(x)^{1-\alpha}\ffi_1(x)^\alpha$ on $[0,\infty)$ (with convention
$0^0:=1$). Then for every $A\in\bM_n^+$ and $B\in\bM_m^+$,
\begin{align*}
&\big\|\Phi(\ffi_\alpha(A))^{1/2}\Psi(\ffi_\alpha(B))
\Phi(\ffi_\alpha(A))^{1/2}\big\|_\infty \\
&\quad\le\big\|\Phi(\ffi_0(A))^{1/2}\Psi(\ffi_0(B))
\Phi(\ffi_0(A))^{1/2}\big\|_\infty^{1-\alpha}
\big\|\Phi(\ffi_1(A))^{1/2}\Psi(\ffi_1(B))
\Phi(\ffi_1(A))^{1/2}\big\|_\infty^\alpha.
\end{align*}
\end{prop}

\begin{proof}
When $f(x):=x^\alpha=f^*(x)$ where $0\le\alpha\le1$, note that
$\ffi_\alpha=\ffi_0\,\sigma_f\,\ffi_1=\ffi_0\,\sigma_f^*\,\ffi_1$. With the same notation
as in the proof of Proposition \ref{P-2.1}, inequality \eqref{F-2.2} is improved in the present
case as
$$
\Psi(\ffi_\alpha(B))\le\gamma_0^{1-\alpha}\gamma_1^\alpha\Phi(\ffi_\alpha(A))^{-1}
$$
for every $\alpha\in[0,1]$. Hence the asserted inequality follows as in the above proof.
\end{proof}

In particular, when $\ffi_0(x)=1$ and $\ffi_1(x)=x$, since $\ffi_0\,\sigma_f\,\ffi_1=f$ and
$\ffi_0\,\sigma_f^*\,\ffi_1=f^*$ in Proposition \ref{P-2.1}, we have

\begin{cor}\label{C-2.3}
Assume that $\Phi(I_n)^{1/2}\Psi(I_m)\Phi(I_n)^{1/2}\le I_l$. If $A\in\bM_n^+$ and
$B\in\bM_m^+$ satisfy $\Phi(A)^{1/2}\Psi(B)\Phi(A)^{1/2}\le I_l$, then
\begin{equation}\label{F-2.4}
\Phi(f^*(A))^{1/2}\Psi(f(B))\Phi(f^*(A))^{1/2}\le I_l
\end{equation}
for every $f\in\OM_{+,1}$, and in particular,
$$
\Phi(A^\alpha)^{1/2}\Psi(B^\alpha)\Phi(A^\alpha)^{1/2}\le I_l,\qquad0\le\alpha\le1.
$$
\end{cor}

\begin{remark}\label{R-2.4}\rm
Assume that both $\Phi$ and $\Psi$ are sub-unital, i.e., $\Phi(I_n)\le I_l$ and
$\Psi(I_m)\le I_l$. If $A\in\bM_n^+$ and $B\in\bM_m^+$ satisfy
$\Phi(A)^{1/2}\Psi(B)\Phi(A)^{1/2}\le I_l$, then one can see \eqref{F-2.4} in a simpler way as
follows: By continuity, one can assume that $\Phi$ is strictly positive and $A>0$; then
$$
\Psi(f(B))\le f(\Psi(B))\le f(\Phi(A)^{-1})=f^*(\Phi(A))^{-1}\le\Phi(f^*(A))^{-1}.
$$
The above first and the last inequalities holds by the Jensen inequality due to
\cite[Theorem 2.1]{Ch} and \cite[Theorem 2.1]{HaPe}. The merit of our method with use of the
operator mean $\sigma_f$ is that it enables us to relax the sub-unitality assumption into
$\Phi(I_n)^{1/2}\Psi(I_m)\Phi(I_n)^{1/2}\le I_l$.
\end{remark}

\section{Log-majorization}

When $\ffi_0$ and $\ffi_1$ are power functions, we can extend Proposition \ref{P-2.2} to the
log-majorization result in the next theorem. For $A\in\bM_n^+$ we write
$\lambda(A)=(\lambda_1(A),\dots,\lambda_n(A))$ for the eigenvalues of $A$ arranged in decreasing
order with multiplicities. Also, for $X\in\bM_n$ let $s(X)=(s_1(X),\dots,s_n(X))$ be the
singular values of $X$ in decreasing order with multiplicities. For two non-negative vectors
$a=(a_1,\dots,a_n)$ and $b=(b_1,\dots,b_n)$ where $a_1\ge\dots\ge a_n\ge0$ and
$b_1\ge\dots\ge b_n\ge0$, the {\it weak log-majorization} (or the {\it log-submajorization})
$a\prec_{w(\log)}b$ means that
\begin{equation}\label{F-3.1}
\prod_{i=1}^ka_i\le\prod_{i=1}^kb_i,\qquad1\le k\le n,
\end{equation}
and the {\it log-majorization} $a\prec_{(\log)}b$ means that $a\prec_{w(\log)}b$ and equality
hold for $k=n$ in \eqref{F-3.1}. On the other hand, the {\it log-supermajorization}
$a\prec^{w(\log)}b$ is defined as
$$
\prod_{i=n+1-k}^na_i\ge\prod_{i=n-k+1}^nb_i,\qquad1\le k\le n.
$$

\begin{thm}\label{T-3.1}
Let $p_0,p_1\in[0,\infty)$ and $0\le\alpha\le1$, and let $p_\alpha:=(1-\alpha)p_0+\alpha p_1$.
Then for every $A\in\bM_n^+$ and $B\in\bM_m^+$,
\begin{align}
&\lambda\bigl(\Phi(A^{p_\alpha})^{1/2}\Psi(B^{p_\alpha})\Phi(A^{p_\alpha})^{1/2}\bigr)
\nonumber\\
&\quad\prec_{w(\log)}
\lambda^{1-\alpha}\bigl(\Phi(A^{p_0})^{1/2}\Psi(B^{p_0})\Phi(A^{p_0})^{1/2}\bigr)
\lambda^\alpha\bigl(\Phi(A^{p_1})^{1/2}\Psi(B^{p_1})\Phi(A^{p_1})^{1/2}\bigr), \label{F-3.2}
\end{align}
or equivalently,
\begin{align}
&s\bigl(\Phi(A^{p_\alpha})^{1/2}\Psi(B^{p_\alpha})^{1/2}\bigr) \nonumber\\
&\qquad\prec_{w(\log)}s^{1-\alpha}\bigl(\Phi(A^{p_0})^{1/2}\Psi(B^{p_0})^{1/2}\bigr)
s^\alpha\bigl(\Phi(A^{p_1})^{1/2}\Psi(B^{p_1})^{1/2}\bigr). \label{F-3.3}
\end{align}
In particular, for every $A,B\in\bM_n^+$,
\begin{equation}\label{F-3.4}
s(A^{p_\alpha}B^{p_\alpha})\prec_{(\log)}
s^{1-\alpha}(A^{p_0}B^{p_0})s^\alpha(A^{p_1}B^{p_1}).
\end{equation}
\end{thm}

\begin{proof}
Let $C^*(I,A)$ be the commutative $C^*$-subalgebra of $\bM_n$ generated by $I,A$. We may
consider, instead of $\Phi$, the composition of the trace-preserving conditional expectation
from $\bM_n$ onto $C^*(I,A)$ and $\Phi|_{C^*(I,A)}:C^*(I,A)\to\bM_d$, which is completely
positive. Hence one can assume that $\Phi$ is completely positive and similarly for $\Psi$. The
weak log-majorization \eqref{F-3.2} means that
\begin{align}
&\prod_{i=1}^k\lambda_i\bigl(\Phi(A^{p_\alpha})^{1/2}\Psi(B^{p_\alpha})
\Phi(A^{p_\alpha})^{1/2}\bigr) \nonumber\\
&\quad\le\prod_{i=1}^k
\lambda_i^{1-\alpha}\bigl(\Phi(A^{p_0})^{1/2}\Psi(B^{p_0})\Phi(A^{p_0})^{1/2}\bigr)
\lambda_i^\alpha\bigl(\Phi(A^{p_1})^{1/2}\Psi(B^{p_1})\Phi(A^{p_1})^{1/2}\bigr) \label{F-3.5}
\end{align}
for every $k=1,\dots,l$. The case $k=1$ is Proposition \ref{P-2.2} in the case where
$\ffi_0(x):=x^{p_0}$ and $\ffi_1(x):=x^{p_1}$ so that $\ffi_\alpha(x)=x^{p_\alpha}$.

Next, for each $k$ with $2\le k\le l$ we consider the $k$-fold tensor product
$$
\Phi^{\otimes k}:\bM_n^{\otimes k}=B((\bC^n)^{\otimes k})\to
\bM_l^{\otimes k}=B((\bC^l)^{\otimes k}),
$$
and similarly for $\Psi^{\otimes k}$. Let $P_\wedge$ be the orthogonal projection from
$(\bC^l)^{\otimes k}$ onto the $k$-fold antisymmetric tensor Hilbert space $(\bC^l)^{\wedge k}$.
Since $\Phi$ and $\Psi$ are assumed completely positive, one can define positive linear maps
\begin{align*}
\Phi^{(k)}&:=P_\wedge\Phi^{\otimes k}(\cdot)P_\wedge:
\bM_n^{\otimes k}\to B((\bC^l)^{\wedge k}), \\
\Psi^{(k)}&:=P_\wedge\Psi^{\otimes k}(\cdot)P_\wedge:
\bM_m^{\otimes k}\to B((\bC^l)^{\wedge k}).
\end{align*}
For every $X\in\bM_n$ we note that
$\Phi^{(k)}(X^{\otimes k})=P_\wedge\Phi(X)^{\otimes k}P_\wedge$ is nothing but the $k$-fold
antisymmetric tensor power $\Phi(X)^{\wedge k}$ of $\Phi(X)$. By applying the case $k=1$ shown
above to $A^{\otimes k}$ and $B^{\otimes k}$ we have
\begin{align*}
&\lambda_1\bigl(\Phi^{(k)}((A^{\otimes k})^{p_\alpha})^{1/2}
\Psi^{(k)}((B^{\otimes k})^{p_\alpha})\Phi^{(k)}((A^{\otimes k})^{p_\alpha})^{1/2}\bigr) \\
&\quad\le\lambda_1^{1-\alpha}\bigl(\Phi^{(k)}((A^{\otimes k})^{p_0})^{1/2}
\Psi^{(k)}((B^{\otimes k})^{p_0})\Phi^{(k)}((A^{\otimes k})^{p_0})^{1/2}\bigr) \\
&\qquad\qquad\lambda_1^\alpha\bigl(\Phi^{(k)}((A^{\otimes k})^{p_1})^{1/2}
\Psi^{(k)}((B^{\otimes k})^{p_1})\Phi^{(k)}((A^{\otimes k})^{p_1})^{1/2}\bigr).
\end{align*}
Since $\Phi^{(k)}((A^{\otimes k})^{p_\alpha})=\Phi(A^{p_\alpha})^{\wedge k}$ and
$\Psi^{(k)}((B^{\otimes k})^{p_\alpha})=\Psi(B^{p_\alpha})^{\wedge k}$, the above left-hand
side is
$$
\lambda_1\Bigl(\bigl(\Phi(A^{p_\alpha})^{1/2}\Psi(B^{p_\alpha})
\Phi(A^{p_\alpha})^{1/2}\bigr)^{\wedge k}\Bigr)
=\prod_{i=1}^k\lambda_i\bigl(\Phi(A^{p_\alpha})^{1/2}\Psi(B^{p_\alpha})
\Phi(A^{p_\alpha})^{1/2}\bigr)
$$
and the right-hand side is
\begin{align*}
&\lambda_1^{1-\alpha}\Bigl(\bigl(\Phi(A^{p_0})^{1/2}\Psi(B^{p_0)})
\Phi(A^{p_0})^{1/2}\bigr)^{\wedge k}\Bigr)
\lambda_1^\alpha\Bigl(\bigl(\Phi(A^{p_1})^{1/2}\Psi(B^{p_1)})
\Phi(A^{p_1})^{1/2}\bigr)^{\wedge k}\Bigr) \\
&\quad=\prod_{i=1}^k\lambda_i^{1-\alpha}\bigl(\Phi(A^{p_0})^{1/2}\Psi(B^{p_0})
\Phi(A^{p_0})^{1/2}\bigr)
\lambda_i^\alpha\bigl(\Phi(A^{p_1})^{1/2}\Psi(B^{p_1})\Phi(A^{p_1})^{1/2}\bigr).
\end{align*}
Hence we have \eqref{F-3.5} for every $k=1,\dots,l$, so \eqref{F-3.2} follows.

Since $\lambda(\Phi(A^{p_\alpha})^{1/2}\Psi(B^{p_\alpha})\Phi(A^{p_\alpha})^{1/2})
=s^2(\Phi(A^{p_\alpha})^{1/2}\Psi(B^{p_\alpha})^{1/2})$, it is clear that \eqref{F-3.2} and
\eqref{F-3.3} are equivalent. When $\Phi=\Psi=\id$ and $A,B$ are replaced with $A^2,B^2$,
\eqref{F-3.3} reduces to \eqref{F-3.4}.
\end{proof}

\begin{remark}\label{R-3.2}\rm
It is not known whether a modification of \eqref{F-3.3}
$$
s(\Phi(A^{p_\alpha})\Psi(B^{p_\alpha}))
\prec_{w(\log)}s^{1-\alpha}(\Phi(A^{p_0})\Psi(B^{p_0}))s^\alpha(\Phi(A^{p_1})\Psi(B^{p_1}))
$$
holds true or not.
\end{remark}

By reducing \eqref{F-3.2} to the case $(p_0,p_1)=(0,1)$ we have

\begin{cor}\label{C-3.3}
Let $0\le\alpha\le1$. Then for every $A\in\bM_n^+$ and $B\in\bM_m^+$,
\begin{align}
&\lambda\bigl(\Phi(A^\alpha)^{1/2}\Psi(B^\alpha)\Phi(A^\alpha)^{1/2}\bigr) \nonumber\\
&\quad\prec_{w(\log)}
\lambda^{1-\alpha}\bigl(\Phi(I_n)^{1/2}\Psi(I_m)\Phi(I_n)^{1/2}\bigr)
\lambda^\alpha\bigl(\Phi(A)^{1/2}\Psi(B)\Phi(A)^{1/2}\bigr). \label{F-3.6}
\end{align}
Consequently, if $\Phi(I_n)^{1/2}\Psi(I_m)\Phi(I_n)^{1/2}\le I_l$, then
\begin{equation}\label{F-3.7}
\lambda\bigl(\Phi(A^\alpha)^{1/2}\Psi(B^\alpha)\Phi(A^\alpha)^{1/2}\bigr)
\prec_{w(\log)}\lambda^\alpha\bigl(\Phi(A)^{1/2}\Psi(B)\Phi(A)^{1/2}\bigr).
\end{equation}
\end{cor}

The last log-majorization with $\Phi=\Psi=\id$ and also \eqref{F-3.4} with $(p_0,p_1)=(0,1)$
give Araki's log-majorization \eqref{F-1.3} or $s(A^\alpha B^\alpha)\prec_{(\log)}s^\alpha(AB)$
for $0\le\alpha\le1$. By letting $\alpha=1/r$ with $r\ge1$ and replacing $A,B$ with $A^r,B^r$
one can rephrase \eqref{F-3.6} as
\begin{align}
&\lambda^r\bigl(\Phi(A)^{1/2}\Psi(B)\Phi(A)^{1/2}\bigr) \nonumber\\
&\quad\prec_{w(\log)}
\lambda^{r-1}\bigl(\Phi(I_n)^{1/2}\Psi(I_m)\Phi(I_n)^{1/2}\bigr)
\lambda\bigl(\Phi(A^r)^{1/2}\Psi(B^r)\Phi(A^r)^{1/2}\bigr) \label{F-3.8}
\end{align}
for all $r\ge1$. Also, when $\Phi(I_n)^{1/2}\Psi(I_m)\Phi(I_n)^{1/2}\le I_l$, \eqref{F-3.7} is
rewritten as
\begin{equation}\label{F-3.9}
\lambda^r\bigl(\Phi(A)^{1/2}\Psi(B)\Phi(A)^{1/2}\bigr)\prec_{w(\log)}
\lambda\bigl(\Phi(A^r)^{1/2}\Psi(B^r)\Phi(A^r)^{1/2}\bigr),\qquad r\ge1.
\end{equation}

A norm $\|\cdot\|$ on $\bM_n$ is called a {\it unitarily invariant norm} (or a {\it symmetric
norm}) if $\|UXV\|=\|X\|$ for all $X,U,V\in\bM_n$ with $U,V$ unitaries.

\begin{cor}\label{C-3.4}
Let $p_0$, $p_1$, and $p_\alpha$ for $0\le\alpha\le1$ be as in Theorem \ref{T-3.1}. Let
$\|\cdot\|$ be any unitarily invariant norm and $r>0$. Then for every $A\in\bM_n^+$ and
$B\in\bM_m^+$,
\begin{align}
&\big\|\,\big|\Phi(A^{p_\alpha})^{1/2}\Psi(B^{p_\alpha})^{1/2}\big|^r\big\| \nonumber\\
&\qquad\le\big\|\,\big|\Phi(A^{p_0})^{1/2}\Psi(B^{p_0})^{1/2}\,\big|^r\big\|^{1-\alpha}
\big\|\,\big|\Phi(A^{p_1})^{1/2}\Psi(B^{p_1})^{1/2}\big|^r\big\|^\alpha. \label{F-3.10}
\end{align}
In particular, for every $A,B\in\bM_n^+$,
$$
\|\,|A^{p_\alpha}B^{p_\alpha}|^r\|
\le\|\,|A^{p_0}B^{p_0}|^r\|^{1-\alpha}\|\,|A^{p_1}B^{p_1}|^r\|^\alpha.
$$
\end{cor}

\begin{proof}
We may assume that $0<\alpha<1$. Let $\psi$ be the symmetric gauge function on $\bR^l$
corresponding to the unitarily invariant norm $\|\cdot\|$, so $\|X\|=\psi(s(X))$ for $X\in\bM_l$.
Recall \cite[IV.1.6]{Bh1} that $\psi$ satisfies the H\"older inequality
$$
\psi(a_1b_1,\dots,a_lb_l)\le
\psi^{1-\alpha}\Bigl(a_1^{1\over1-\alpha},\dots,a_l^{1\over1-\alpha}\Bigr)
\psi^\alpha\Bigl(b_1^{1\over\alpha},\dots,b_l^{1\over\alpha}\Bigr)
$$
for every $a,b\in[0,\infty)^l$. Also, it is well-known (see, e.g.,
\cite[Proposition 4.1.6 and Lemma 4.4.2]{Hi2}) that $a\prec_{w(\log)}b$ implies the weak
majorization $a\prec_wb$ and so $\psi(a)\le\psi(b)$. Hence it follows from the weak
log-majorization in \eqref{F-3.3} that
\begin{align*}
&\big\|\,\big|\Phi(A^{p_\alpha})^{1/2}\Psi(B^{p_\alpha})^{1/2}\big|^r\big\| \\
&\quad=\psi\Bigl(s^r\bigl(\Phi(A^{p_\alpha})^{1/2}\Psi(B^{p_\alpha})^{1/2}\bigr)\Bigr) \\
&\quad\le\psi\Bigl(s^{(1-\alpha)r}\bigl(\Phi(A^{p_0})^{1/2}\Psi(B^{p_0})^{1/2}\bigr)
s^{\alpha r}\bigl(\Phi(A^{p_1})^{1/2}\Psi(B^{p_1})^{1/2}\bigr)\Bigr) \\
&\quad\le\psi^{1-\alpha}\Bigl(s^r\bigl(\Phi(A^{p_0})^{1/2}\Psi(B^{p_0})^{1/2}\bigr)\Bigr)
\psi^\alpha\Bigl(s^r\bigl(\Phi(A^{p_1})^{1/2}\Psi(B^{p_1})^{1/2}\bigr)\Bigr) \\
&\quad\le\big\|\,\big|\Phi(A^{p_0})^{1/2}\Psi(B^{p_0})^{1/2}\,\big|^r\big\|^{1-\alpha}
\big\|\,\big|\Phi(A^{p_1})^{1/2}\Psi(B^{p_1})^{1/2}\big|^r\big\|^\alpha.
\end{align*}
\end{proof}

The norm inequality in \eqref{F-3.10} is a kind of the H\"older type inequality, showing the
log-convexity of the function
$$
p\in[0,\infty)\longmapsto\big\|\,\big|\Phi(A^p)^{1/2}\Psi(B^p)^{1/2}\big|^r\big\|.
$$

\begin{cor}\label{C-3.5}
Let $\|\cdot\|$ be a unitarily invariant norm. If
$\Phi(I_n)^{1/2}\Psi(I_m)\Phi(I_n)^{1/2}\le I_l$, then for every $A\in\bM_n^+$ and $B\in\bM_m^+$,
$$
\big\|\bigl\{\Phi(A^p)^{1/2}\Psi(B^p)\Phi(A^p)^{1/2}\bigr\}^{1/p}\big\|\le
\big\|\bigl\{\Phi(A^q)^{1/2}\Psi(B^q)\Phi(A^q)^{1/2}\bigr\}^{1/q}\big\|
\quad\mbox{if $0<p\le q$}.
$$
Furthermore, if $\Phi$ and $\Psi$ are unital and $A,B>0$, then 
$$
\big\|\bigl\{\Phi(A^p)^{1/2}\Psi(B^p)\Phi(A^p)^{1/2}\bigr\}^{1/p}\big\|
$$
decreases to $\|\exp\{\Phi(\log A)+\Psi(\log B)\}\|$ as $p\searrow0$.
\end{cor}

\begin{proof}
Let $0<p\le q$. By applying \eqref{F-3.7} to $A^q$, $B^q$ and $\alpha=p/q$ we have
$$
\lambda^{1/p}\bigl(\Phi(A^p)^{1/2}\Psi(B^p)\Phi(A^p)^{1/2}\bigr)
\prec_{w(\log)}\lambda^{1/q}\bigl(\Phi(A^q)^{1/2}\Psi(B^q)\Phi(A^q)^{1/2}\bigr),
$$
which implies the desired norm inequality. Under the additional assumptions on $\Phi,\Psi$ and
$A,B$ as stated in the corollary, the proof of the limit formula is standard with use of
$$
\Phi(A^p)^{1/2}=I_l+{p\over2}\,\Phi(\log A)+o(p),\quad
\Psi(B^p)=I_l+p\Psi(\log B)+o(p).\
$$
as $p\to0$.
\end{proof}

\begin{remark}\label{R-3.6}\rm
When $\Phi=\id$ and $\Psi=Z^*\cdot Z$ with a contraction $Z\in\bM_n$, it follows from
\eqref{F-3.9} that, for every $A,B\in\bM_n^+$,
\begin{equation}\label{F-3.11}
\lambda^r\bigl(A^{1/2}Z^*BZA^{1/2}\bigr)\prec_{w(\log)}
\lambda\bigl(A^{r/2}Z^*B^rZA^{r/2}\bigr),\qquad r\ge1,
\end{equation}
which is \cite[Corollary 2.3]{BL}. Although the form of \eqref{F-3.9} is seemingly more general
than that of \eqref{F-3.11}, it is in fact easy to see that \eqref{F-3.9} follows from
\eqref{F-3.11} conversely. Indeed, we may assume as in the proof of Theorem \ref{T-3.1} that
$\Phi$ and $\Psi$ are completely positive. Then, via the Stinespring representation (see, e.g.,
\cite[Theorem 3.1.2]{Bh2}), we may further assume that $\Phi=V^*\cdot V$ with an operator
$V:\bC^l\to\bC^n$ and $\Psi=W^*\cdot W$ with an operator $W:\bC^l\to\bC^m$. The assumption
$\Phi(I)^{1/2}\Psi(I)\Phi(I)^{1/2}\le I$ is equivalent to $\|WV^*\|_\infty\le1$. One can see
that
$$
\Phi(A^r)^{1/2}\Psi(B^r)\Phi(A^r)^{1/2}=(V^*A^rV)^{1/2}(W^*B^rW)(V^*A^rV)^{1/2}
$$
is unitarily equivalent to $A^{r/2}VW^*B^rWV^*A^{r/2}$, and thus \eqref{F-3.11} implies
\eqref{F-3.9}. Here, it should be noted that the proof of \eqref{F-3.11} in \cite{BL} is valid
even though $Z=WV^*$ is an $m\times n$ (not necessarily square) matrix. In this way, the
log-majorization in \eqref{F-3.9} is equivalent to \cite[Corollary 2.3]{BL}. Similarly,
Corollary \ref{C-3.5} is equivalent to \cite[Corollary 2.2]{BL}. The author is indebted to
J.-C.~Bourin for the remark here.
\end{remark}

\section{More inequalities for operator means}

The log-majorization obtained in \cite{AH} for the weighted geometric means says that, for
every $0\le\alpha\le1$ and every $A,B\in\bM_n^+$,
\begin{equation}\label{F-4.1}
\lambda(A^r\,\#_\alpha\,B^r)\prec_{(\log)}\lambda^r(A\,\#_\alpha\,B),
\qquad r\ge1,
\end{equation}
or equivalently,
$$
\lambda^q(A\,\#_\alpha\,B)\prec_{(\log)}\lambda(A^q\,\#_\alpha\,B^q),
\qquad0\le q\le1.
$$
The essential first step to prove this is the operator norm inequality
$$
\|A^r\,\#_\alpha\,B^r\|_\infty\le\|A\,\#_\alpha\,B\|_\infty^r,\qquad r\ge1,
$$
which is equivalent to that $A\,\#_\alpha\,B\le I$ $\Rightarrow$ $A^r\,\#_\alpha\,B^r\le I$ for
all $r\ge1$. By taking the inverse when $A,B>0$, this is also equivalent to that
$A\,\#_\alpha\,B\ge I$ $\Rightarrow$ $A^r\,\#_\alpha\,B^r\ge I$ for all $r\ge1$. The last
implication was recently extended in \cite[Lemmas 2.1, 2.2]{Wa} to the assertion stating the
equivalence between the following two conditions for $f\in\OM_{+,1}$:
\begin{itemize}
\item[(i)] $f(x)^r\le f(x^r)$ for all $x\ge0$ and $r\ge1$;
\item[(ii)] for every $A,B\in\bM_n^+$, $A\,\sigma_f\,B\ge I$ $\Rightarrow$
$A^r\,\sigma_f\,B^r\ge I$ for all $r\ge1$.
\end{itemize}
We note that the above conditions are also equivalent to
\begin{itemize}
\item[(iii)] for every $A,B\in\bM_n^+$,
$$
\lambda_n(A^r\,\sigma_f\,B^r)\ge\lambda_n^r(A\,\sigma_f\,B),\qquad r\ge1;
$$
or equivalently, for every $A,B\in\bM_n^+$,
$$
\lambda_n(A^q\,\sigma_f\,B^q)\le\lambda_n^q(A\,\sigma_f\,B),\qquad0<q\le1.
$$
\end{itemize}

The next proposition extends the above result to the form involving positive linear maps.
Below let $\Phi$ and $\Psi$ be positive linear maps as before.

\begin{prop}\label{P-4.1}
Assume that $f\in\OM_{+,1}$ satisfies the above condition {\rm(i)}. Then for every $A\in\bM_n^+$
and $B\in\bM_m^+$,
\begin{equation}\label{F-4.2}
\bigl(\max\{\|\Phi(I_n)\|_\infty,\|\Psi(I_m)\|_\infty\}\bigr)^{r-1}
\lambda_l\bigl(\Phi(A^r)\,\sigma_f\,\Psi(B^r)\bigr)
\ge\lambda_l^r\bigl(\Phi(A)\,\sigma_f\,\Psi(B)\bigr)
\end{equation}
for all $r\ge1$.
\end{prop}

\begin{proof}
By continuity we may assume that $\Phi$ and $\Psi$ are strictly positive. Let $0<q\le1$.  Since
$\Phi(I_n)^{-1/2}\Phi(\cdot)\Phi(I_n)^{-1/2}$ is a unital positive linear map, it is well-known
\cite[Proposition 2.7.1]{Bh2} that
$$
\Phi(I_n)^{-1/2}\Phi(A^q)\Phi(I_n)^{-1/2}\le
\bigl(\Phi(I_n)^{-1/2}\Phi(A)\Phi(I_n)^{-1/2}\bigr)^q
$$
so that
\begin{align*}
\Phi(A^q)&\le\Phi(I_n)^{1/2}\bigl(\Phi(I_n)^{-1/2}\Phi(A)\Phi(I_n)^{-1/2}\bigr)^q
\Phi(I_n)^{1/2} \\
&=\Phi(I_n)\,\#_q\,\Phi(A)\le(\|\Phi(I_n)\|_\infty I_n)\,\#_q\,\Phi(A) \\
&=\|\Phi(I_n)\|_\infty^{1-q}\Phi(A)^q
\end{align*}
and similarly
$$
\Psi(B^q)\le\|\Psi(I_m)\|_\infty^{1-q}\Psi(B)^q.
$$
By the joint monotonicity of $\sigma_f$ we have
\begin{align}
\Phi(A^q)\,\sigma_f\,\Psi(B^q)
&\le\bigl(\|\Phi(I_n)\|_\infty^{1-q}\Phi(A)^q\bigr)\,\sigma_f\,
\bigl(\|\Psi(I_m)\|_\infty^{1-q}\Psi(B)^q\bigr) \nonumber\\
&\le\bigl(\max\{\|\Phi(I_n)\|_\infty,\|\Psi(I_m)\|_\infty\}\bigr)^{1-q}
\bigl(\Phi(A)^q\,\sigma_f\,\Psi(B)^q\bigr). \label{F-4.3}
\end{align}
Therefore,
\begin{align*}
\lambda_l\bigl(\Phi(A^q)\,\sigma_f\,\Psi(B^q)\bigr)
&\le\bigl(\max\{\|\Phi(I_n)\|_\infty,\|\Psi(I_m)\|_\infty\}\bigr)^{1-q}
\lambda_l\bigl(\Phi(A)^q\,\sigma_f\,\Psi(B)^q\bigr) \\
&\le\bigl(\max\{\|\Phi(I_n)\|_\infty,\|\Psi(I_m)\|_\infty\}\bigr)^{1-q}
\lambda_l^q\bigl(\Phi(A)\,\sigma_f\,\Psi(B)\bigr)
\end{align*}
by using the property (iii) above. Now, for $0<r\le1$ let $q:=1/r$. By replacing $A,B$ with
$A^r,B^r$, respectively, we obtain
$$
\lambda_l\bigl(\Phi(A)\,\sigma_f\,\Psi(B)\bigr)
\le\bigl(\max\{\|\Phi(I_n)\|_\infty,\|\Psi(I_m)\|_\infty\}\bigr)^{1-{1\over r}}
\lambda_l^{1/r}\bigl(\Phi(A^r)\,\sigma_f\,\Psi(B^r)\bigr),
$$
which yields \eqref{F-4.2}.
\end{proof}

When $\sigma_f$ is the weighted geometric mean $\#_\alpha$, one can improve Proposition
\ref{P-4.1} to the log-supermajorization result as follows:

\begin{prop}\label{P-4.2}
Let $0\le\alpha\le1$. Then for every $A\in\bM_n^+$ and $B\in\bM_m^+$.
\begin{equation}\label{F-4.4}
\bigl(\|\Phi(I_n)\|_\infty\,\#_\alpha\,\|\Psi(I_m)\|_\infty\bigr)^{r-1}
\lambda\bigl(\Phi(A^r)\,\#_\alpha\,\Psi(B^r)\bigr)
\prec^{w(\log)}\lambda^r\bigl(\Phi(A)\,\#_\alpha\,\Psi(B)\bigr)
\end{equation}
for all $r\ge1$. Consequently, if $\|\Phi(I_n)\|_\infty\,\#_\alpha\,\|\Psi(I_m)\|_\infty\le1$,
then
$$
\lambda\bigl(\Phi(A^r)\,\#_\alpha\,\Psi(B^r)\bigr)
\prec^{w(\log)}\lambda^r\bigl(\Phi(A)\,\#_\alpha\,\Psi(B)\bigr),\qquad r\ge1.
$$
\end{prop}

\begin{proof}
When $\sigma_f=\#_\alpha$, inequality \eqref{F-4.3} is improved as
$$
\Phi(A^q)\,\#_\alpha\,\Psi(B^q)
\le\bigl(\|\Phi(I_n)\|_\infty\,\#_\alpha\,\|\Psi(I_m)\|_\infty\bigr)^{1-q}
\bigl(\Phi(A)^q\,\#_\alpha\,\Psi(B)^q\bigr)
$$
for $0<q\le1$, and hence \eqref{F-4.2} is improved as
$$
\bigl(\|\Phi(I_n)\|_\infty\,\#_\alpha\,\|\Psi(I_m)\|_\infty\bigr)^{r-1}
\lambda_l\bigl(\Phi(A^r)\,\#_\alpha\,\Psi(B^r)\bigr)
\ge\lambda_l^r\bigl(\Phi(A)\,\#_\alpha\,\Psi(B)\bigr)
$$
for all $r\ge1$. One can then prove the asserted log-supermajorization result in the same way
as in the proof of Theorem \ref{T-3.1} with use of the antisymmetric tensor power technique,
where the identity $\lambda_l(X^{\wedge k})=\prod_{i=l-k+1}^l\lambda_i(X)$ for $X\in\bM_l^+$ is
used instead of $\lambda_1(X^{\wedge k})=\prod_{i=1}^k\lambda_i(X)$ in the previous proof. The
details may be omitted here.
\end{proof}

In particular, when $\Phi=\Psi=\id$, \eqref{F-4.4} reduces to \eqref{F-4.1} since,
for $A,B>0$, the log-supermajorization
$\lambda(A^r\,\#_\alpha\,B^r)\prec^{w(\log)}\lambda^r(A\,\#_\alpha\,B)$ implies the
log-majorization \eqref{F-4.1}.

The notion of symmetric anti-norms was introduced in \cite{BH1,BH2} with the notation
$\|\cdot\|_!$. Recall that a non-negative continuous functional $\|\cdot\|_!$ on $\bM_n^+$ is
called a {\it symmetric anti-norm} if it is positively homogeneous, superadditive (instead of
subadditive in case of usual norms) and unitarily invariant. Among others, a symmetric anti-norm
is typically defined associated with a symmetric norm $\|\cdot\|$ on $\bM_n$ and $p>0$ in such
a way that, for $A\in\bM_n^+$,
$$
\|A\|_!:=\begin{cases}\|A^{-p}\|^{-1/p} & \text{if $A$ is invertible}, \\
0 & \text{otherwise}.
\end{cases}
$$
A symmetric anti-norm defined in this way is called a {\it derived anti-norm}, see
\cite[Proposition 4.6]{BH2}. By Lemma \cite[Lemma 4.10]{BH2}, similarly to Corollary \ref{C-3.5},
we have

\begin{cor}\label{C-4.3}
Let $0\le\alpha\le1$ and assume that $\|\Phi(I_n)\|_\infty\,\#_\alpha\,\|\Psi(I_m)\|_\infty\le1$.
Then for every $A\in\bM_n^+$ and $B\in\bM_m^+$ and for any derived anti-norm $\|\cdot\|_!$ on
$\bM_l^+$,
$$
\big\|\{\Phi(A^p)\,\#_\alpha\,\Psi(B^p)\}^{1/p}\big\|_!\ge
\big\|\{\Phi(A^q)\,\#_\alpha\,\Psi(B^q)\}^{1/q}\big\|_!,\quad\mbox{if $0<p\le q$}.
$$
\end{cor}

\begin{problem}\label{Q-4.4}\rm
It seems that our generalization of Ando-Hiai type log-majorization is not so much completed as
that of Araki's log-majorization in Section 3. Although the form of \eqref{F-4.4} bears some
resemblance to that of \eqref{F-3.8}, they have also significant differences. For one thing,
$\prec^{w(\log)}$ arises in \eqref{F-4.4} while $\prec_{w(\log)}$ in \eqref{F-3.8}, which should
be reasonable since the directions of log-majorization are opposite between them. For another,
the factor $\bigl(\|\Phi(I_n)\|_\infty\,\#_\alpha\,\|\Psi(I_m)\|_\infty\bigr)^{r-1}$ in
\eqref{F-4.4} is apparently much worse than
$\lambda^{r-1}\bigl(\Phi(I_n)^{1/2}\Psi(I_m)\Phi(I_n)^{1/2}\bigr)$ in \eqref{F-3.8}. One might
expect the better factor $\|\Phi(I_n)\,\#_\alpha\,\Psi(I_m)\|_\infty^{r-1}$ or even
$\lambda^{r-1}(\Phi(I_n)\,\#_\alpha\,\Psi(I_m))$. Indeed, a more general interesting problem is
the $\#_\alpha$-version of \eqref{F-3.2}, i.e., for $p_0,p_1\ge0$, $0\le\theta\le1$ and
$p_\theta:=(1-\theta)p_0+\theta p_1$,
$$
\lambda^{1-\theta}(\Phi(A^{p_0})\,\#_\alpha\,\Psi(B^{p_0}))
\lambda^\theta(\Phi(A^{p_1})\,\#_\alpha\,\Psi(B^{p_1}))\prec^{w(\log)}
\lambda(\Phi(A^{p_\theta})\,\#_\alpha\,\Psi(B^{p_\theta}))\,?
$$
When $\Phi=\Psi=\id$, the problem becomes
\begin{equation}\label{F-4.5}
\lambda^{1-\theta}(A^{p_0}\,\#_\alpha\,B^{p_0})
\lambda^\theta(A^{p_1}\,\#_\alpha\,B^{p_1})\prec_{(\log)}
\lambda(A^{p_\theta}\,\#_\alpha\,B^{p_\theta})\,?
\end{equation}
\end{problem}

\begin{example}\rm
Here is a sample computation of the last problem for $A,B$ are $2\times2$ and $\alpha=2$.
Thanks to continuity and homogeneity, we may assume that $A,B\in\bM_2^+$ are invertible with
determinant $1$. So we write $A=aI+\bx\cdot\sigma$ and $B=bI+\by\cdot\sigma$ with $a,b>0$,
$\bx,\by\in\bR^3$, $\det A=a^2-|\bx|^2=1$ and $\det B=b^2-|\by|^2=1$, where
$|\bx|^2:=x_1^2+x_2^2+x_3^2$ and $\bx\cdot\sigma:=x_1\sigma_1+x_2\sigma_2+x_3\sigma_3$ with
Pauli matrices $\sigma_i$, i.e., $\sigma_1=\begin{bmatrix}0&1\\1&0\end{bmatrix}$,
$\sigma_2=\begin{bmatrix}0&-i\\i&0\end{bmatrix}$,
$\sigma_3=\begin{bmatrix}1&0\\0&-1\end{bmatrix}$. For \eqref{F-4.5} in this situation, it
suffices, thanks to \cite[Proposition 3.11]{Mo} (also \cite[Proposition 4.1.12]{Bh2}), to show
that
\begin{equation}\label{F-4.6}
p\ge0\longmapsto\lambda_1\Biggl({A^p+B^p\over\sqrt{\det(A^p+B^p)}}\Biggr)
=\biggl({\lambda_1(A^p+B^p)\over\lambda_2(A^p+B^p)}\biggr)^{1/2}
\end{equation}
is a log-concave function. Let $e^\alpha=a+|\bx|$ and $e^\beta=b+|\by|$, so
$e^{-\alpha}=a-|\bx|$, $|\bx|=\sinh\alpha$, and similarly for $|\by|$. Then a direct
computation yields
$$
A^p+B^p=(\cosh(\alpha p)+\cosh(\beta p))I
+\biggl[{\sinh(\alpha p)\over\sinh\alpha}\,\bx+{\sinh(\beta p)\over\sinh\beta}\,\by\biggr]
\cdot\sigma,
$$
whose eigenvalues are
$$
\cosh(\alpha p)+\cosh(\beta p)\pm\bigl[\sinh^2(\alpha p)+\sinh^2(\beta p)
+2c\sinh(\alpha p)\sinh(\beta p)\bigr]^{1/2}
$$
with $c:={\bx\cdot\by\over|\bx|\,|\by|}\in[-1,1]$. Although numerical computations say that
\eqref{F-4.6} is a log-concave function of $p\ge0$ for any $\alpha,\beta\ge0$ and $c\in[-1,1]$,
it does not seem easy to give a rigorous proof.
\end{example}

In the rest of the paper we present one more log-majorization result. Let $E\in\bM_n$ be an
orthogonal projection with $\dim E=l$. A particular case of \eqref{F-3.2} is
\begin{equation}\label{F-4.7}
\lambda(EA^{(1-\theta)p_0+\theta p_1}E)\prec_{w(\log)}
\lambda^{1-\theta}(EA^{p_0}E)\lambda^\theta(EA^{p_1}E),\qquad0\le\theta\le1
\end{equation}
for every $A\in\bM_n^+$. As a complementary version of this we show the following:

\begin{prop}\label{P-4.6}
Let $p_0,p_1\ge0$ and $0\le\theta\le1$. Then for every $\alpha\in(0,1]$ and $A\in\bM_n^+$,
\begin{equation}\label{F-4.8}
\bigl(\lambda_i(A^{(1-\theta)p_0+\theta p_1}\,\#_\alpha\,E)\bigr)_{i=1}^l\prec^{w(\log)}
\bigl(\lambda_i^{1-\theta}(A^{p_0}\,\#_\alpha\,E)
\lambda_i^\theta(A^{p_1}\,\#_\alpha\,E)\bigr)_{i=1}^l.
\end{equation}
\end{prop}

The form of this log-majorization is similar to that of the problem \eqref{F-4.5}. Although
the directions of those are opposite, there is no contradiction between those two; indeed,
the log-majorization of \eqref{F-4.5} is taken for matrices in $\bM_n^+$ while that of
\eqref{F-4.8} is for $l\times l$ matrices restricted to the range of $E$.

First, we give a lemma in a setting of more general operator means. Let $f$ be an operator
monotone function on $[0,\infty)$ such that $f(0)=0$, and let $\sigma_f$ be the operator mean
corresponding to $f$ due to Kubo-Ando theory. An operator monotone function dual to $f$ is
defined by $f^\perp(x):=x/f(x)$, $x>0$, and $f^\perp(0):=\lim_{x\searrow0}f^\perp(x)$.

\begin{lemma}\label{L-4.7}
Let $f$ and $f^\perp$ be as stated above. Then for every $A\in\bM_n^+$ with $A>0$,
$$
A\,\sigma_f\,E=(Ef^\perp(EA^{-1}E)E)^{-1},
$$
where the inverse in the right-hand side is defined on the range of $E$ (i.e., in the sense of
generalized inverse).
\end{lemma}

\begin{proof}
For $k=0,1,2,\dots$ we have
$$
A^{-1/2}E(EA^{-1}E)^kEA^{-1/2}=(A^{-1/2}EA^{-1/2})^{k+1}.
$$
Define a function $\widehat f$ on $[0,\infty)$ by $\widehat f(x):=f(x)/x$ for $x>0$ and
$\widehat f(0):=0$. Note that the eigenvalues of $EA^{-1}E$ and those of $A^{-1/2}EA^{-1/2}$ are
the same including multiplicities. By approximating $\widehat f$ by polynomials on the
eigenvalues of $EA^{-1}E$, we have
$$
A^{-1/2}E\widehat f(EA^{-1}E)EA^{-1/2}=A^{-1/2}EA^{-1/2}\widehat f(A^{-1/2}EA^{-1/2})
=f(A^{-1/2}EA^{-1/2})
$$
since the assumption $f(0)=0$ implies that $f(x)=x\widehat f(x)$ for all $x\in[0,\infty)$.
Therefore,
$$
E\widehat f(EA^{-1}E)E=A^{1/2}f(A^{-1/2}EA^{-1/2})A^{1/2}=A\,\sigma_f\,E.
$$
Moreover, it is easy to verify that $(Ef^\perp(EA^{-1}E)E)^{-1}=E\widehat f(EA^{-1}E)E$.
\end{proof}

\noindent
{\it Proof of Proposition \ref{P-4.6}.}\enspace
Since the result is trivial when $\alpha=1$, we may assume that $0<\alpha<1$. Moreover, we may
assume by continuity that $A$ is invertible. When $f(x)=x^\alpha$, note that $\sigma_f=\#_\alpha$
and $f^\perp(x)=x^{1-\alpha}$. Hence by Lemma \ref{L-4.7} we have
$$
A^p\,\#_\alpha\,E=(EA^{-p}E)^{\alpha-1},\qquad p\ge0,
$$
where $(E\cdot E)^{\alpha-1}$ is defined on the range of $E$. This implies that, for every
$k=1,\dots,l$,
$$
\prod_{i=l-k+1}^l\lambda_i(A^p\,\#_\alpha\,E)
=\Biggl(\prod_{i=1}^k\lambda_i(EA^{-p}E)\Biggr)^{\alpha-1}
$$
so that \eqref{F-4.8} immediately follows from \eqref{F-4.7} applied to $A^{-1}$.\qed

\bigskip
Similarly to Corollary \ref{C-3.4}, by Proposition \ref{P-4.6} and
\cite[Lemma 4.10 and (4.4)]{BH2} we see that if $A\in\bM_n^+$ and $\|\cdot\|_!$ is a derived
anti-norm on $\bM_l^+$, then $\|A^p\,\#_\alpha\,E\|_!$ is a log-concave function of $p\ge0$,
where $A^p\,\#_\alpha\,E$ is considered as an $l\times l$ matrix restricted to the range of $E$.

\subsection*{Acknowledgments}
This research was supported in part by Grant-in-Aid for Scientific Research (C)21540208.


\begin{thebibliography}{99}

\bibitem{An1}
T. Ando, Concavity of certain maps on positive definite matrices and applications to Hadamard
products, Linear Algebra Appl. 26 (1979), 203--241.

\bibitem{An2}
T. Ando, Majorization, doubly stochastic matrices, and comparison of eigenvalues,
{\it Linear Algebra Appl.} {\bf 118} (1989), 163--248 .

\bibitem{AH}
T. Ando and F. Hiai, Log majorization and complementary Golden-Thompson type
inequalities, {\it Linear Algebra Appl.} {\bf 197} (1994), 113--131.

\bibitem{Ar}
H. Araki, On an inequality of Lieb and Thirring, {\it Lett. Math. Phys.} {\bf 19}
(1990), 167--170.

\bibitem{Bh1}
R. Bhatia, {\it Matrix Analysis}, Springer, New York, 1996.

\bibitem{Bh2}
R. Bhatia, {\it Positive Definite Matrices}, Princeton University Press, Princeton, 2007.

\bibitem{BH1}
J.-C. Bourin and F. Hiai,
Norm and anti-norm inequalities for positive semi-definite matrices,
{\it Internat. J. Math.} {\bf 22} (2011), 1121--1138.

\bibitem{BH2}
J.-C. Bourin and F. Hiai,
Jensen and Minkowski inequalities for operator means and anti-norms,
{\it Linear Algebra Appl.} {\bf 456} (2014), 22--53.

\bibitem{BL}
J.-C. Bourin and E.-Y. Lee, Matrix inequalities of Araki-H\"older type, Preprint (2015).
arXiv:1511.06977

\bibitem{Ch}
M.-D. Choi, A Schwarz inequality for positive linear maps on $C^*$-algebras,
{\it Illinois J. Math.} {\bf 18} (1974), 565--574.

\bibitem{HaPe}
F. Hansen and G. K. Pedersen, Jensen's inequality for operators and L\"owner's theorem,
{\it Math. Ann.} {\bf 258} (1982), 229--241.

\bibitem{Hi1}
F. Hiai, Log-majorizations and norm inequalities for exponential operators, 
in {\it Linear Operators}, J. Janas, F. H. Szafraniec and J. Zem\'anek (eds.),
Banach Center Publications, Vol. 38, 1997, pp. 119--181.

\bibitem{Hi2}
F. Hiai, Matrix Analysis: Matrix Monotone Functions, Matrix Means, and Majorization,
{\it Interdisciplinary Information Sciences} {\bf 16} (2010), 139--248.

\bibitem{HP}
F. Hiai and D. Petz, The Golden-Thompson trace inequality is complemented,
{\it Linear Algebra Appl.} {\bf 181} (1993), 153--185.

\bibitem{KA}
F. Kubo and T. Ando, Means of positive linear operators,
{\it Math. Ann.} {\bf 246} (1980), 205--224.

\bibitem{LT}
E. H. Lieb and W.\ Thirring, Inequalities for the moments of the eigenvalues of the
Schr\"odinger Hamiltonian and their relation to Sobolev inequalities,
in: Studies in Mathematical Physics, E. H. Lieb, B. Simon, and A. S. Wightman (eds.),
Princeton Univ. Press, Princeton, 1976, pp. 301--302.

\bibitem{Mo}
M. Moakher, A differential geometric approach to the geometric mean of symmetric
positive definite matrices, {\it SIAM J. Matrix Anal. Appl.} {\bf 26} (2005), 735--747.

\bibitem{PW}
W. Pusz and S. L. Woronowicz, Functional calculus for sesquilinear forms and the
purification map, {\it Rep. Math. Phys.} {\bf 8} (1975), 159--170.

\bibitem{Wa}
S. Wada, Some ways of constructing Furuta-type inequalities,
{\it Linear Algebra Appl.} {\bf 457} (2014) 276--286.

\end{thebibliography}
\end{document}